\documentclass[a4paper,reqno,11pt]{amsart}
\usepackage{epsfig}
\usepackage{amsmath}
\usepackage[pdftex]{hyperref}
\usepackage{amsfonts}
\usepackage{mathrsfs}
\usepackage{mathtools}
\usepackage{amsmath,amscd}
\usepackage{fullpage}
\usepackage{amssymb}
\usepackage{amsthm}
\usepackage{mathrsfs}
\usepackage{graphicx}
\usepackage{xypic}

\theoremstyle{plain}
\newtheorem{theorem}{Theorem}[section]

\newtheorem{definition}[theorem]{Definition}
\newtheorem{lemma}[theorem]{Lemma}
\newtheorem{corollary}[theorem]{Corollary}
\newtheorem{proposition}[theorem]{Proposition}

\newtheorem{remark}[theorem]{Remark}

\newtheorem*{thma}{Theorem A}
\newtheorem*{thmb}{Theorem B}

\theoremstyle{definition}

\newcommand{\C}{\mathbb{C}}
\newcommand{\R}{\mathbb{R}}
\newcommand{\Z}{\mathbb{Z}}
\newcommand{\Sp}{\mathbb{S}}
\newcommand{\Hy}{\mathbb{H}}
\newcommand{\Ker}{\mathit{Ker}}
\newcommand{\colim}{\mathit{colim}}
\newcommand{\PP}{\mathcal{P}}
\newcommand{\CC}{\mathcal{C}}

\numberwithin{equation}{section}
\begin{document}
	\title{ Inertia Groups and Smooth Structures on Quaternionic Projective Spaces}
	\vspace{2cm}
	
	 \author{Samik Basu}
\address{Stat-Math Unit\\ Indian Statistical Institute \\ Kolkata -- 700108 \\ India.}
\email{samik.basu2@gmail.com; samikbasu@isical.ac.in}
	
\author{Ramesh Kasilingam}
	
	\address{Department of Mathematics,
		Indian Institute Of Technology, Chennai-600036, India}
\email{rameshkasilingam.iitb@gmail.com  ; rameshk@iitm.ac.in}
	
	\date{}
	
	\maketitle

	\begin{abstract}
This paper deals with certain results on the number of smooth structures on quaternionic projective spaces, obtained through the computation of inertia group and its analogues, which in turn are computed using techniques from stable homotopy theory. We show that the concordance inertia group is trivial in dimension 20, but there are many examples in high dimensions where the concordance inertia group is non-trivial. We extend these to computations of concordance classes of smooth structures. These have applications to $3$-sphere actions on homotopy spheres and tangential homotopy structures. 
	\end{abstract}

	\section{Introduction}

The study of exotic structures on manifolds\footnote{In this paper all manifolds will be closed smooth, oriented and connected, and all homeomorphisms and diffeomorphisms are assumed to preserve orientation, unless otherwise stated.} is a topic of fundamental interest in differential topology. The first such example comes up in the celebrated paper of Milnor \cite{Mil56}, in examples of manifolds that are homeomorphic to $S^7$ but not diffeomorphic. This inspires the definition of the set $\Theta_n$ of differentiable manifolds homeomorphic to $S^n$. These form a group under connected sum and are related to stable homotopy groups in \cite{KM63}. 

A possible way to change the smooth structure on an oriented smooth manifold $M^n$ without changing the homeomorphism type is by taking connected sum with an exotic sphere.  This induces an action of $\Theta_n$ on the set of smooth structures on $M$. The stabilizer of $M$ under this action is the inertia group $I(M)$ of $M$. 

This paper deals with computations in the inertia groups of quaternionic projective spaces, and associated computations for smooth structures. Computations in the inertia group are known for certain products of spheres \cite{Sch71}, $3$-sphere bundles over $S^4$ \cite{Tam62}, low dimensional complex projective spaces \cite{Kaw68, BK}. However, there is no systematic approach for computing inertia groups in general, and many problems are still open. 

One also makes certain analogous definitions for an oriented smooth manifold $M$, such as the homotopy inertia group $I_h(M)$ and the concordance inertia group $I_c(M)$. $I_h(M)$ (respectively $I_c(M)$) consists of those $\Sigma \in I(M)$  for which the diffeomorphism $M\# \Sigma\cong M$ is homotopic (respectively concordant) to the canonical homeomorphism $h_{can}:M\#\Sigma\to M$. These groups are the same for spheres and complex projective spaces, but for $\Hy P^n$, the homotopy inertia group and the concordance inertia group are the same, while the inertia group may be different from these. It follows from \cite{AF04} that the concordance inertia group is trivial for $n\leq 4$, while \cite{KS07} demonstrates that the inertia group of $\Hy P^2$ is non-trivial. In this paper we prove 
\begin{thma}
\begin{itemize}
\item[(a)] $I_c(\Hy P^5)=0$.
\item[(b)] There are infinitely many values of $n$ for which there exist non-trivial elements in the inertia group $I(\Hy P^n)$. 
\end{itemize}
\end{thma}
(cf. Theorem \ref{exotic} and Theorem \ref{inerhigh} ).

We follow up computations in the inertia group by computing the group $\mathcal{C}(\Hy P^n)$ of concordance classes of smooth structures on $\Hy P^n$ for $n\leq 5$. For $n=2$, this was computed in \cite{KS07}, where it was shown that $\mathcal{C}(\Hy P^2)=\{[(\Hy P^2,Id)], [(\Hy P^2\#\Sigma^8,h_{can})]\}$, where $\Sigma^8$ is the exotic $8$-sphere. For $3\leq n\leq 5$, we prove

\begin{thmb}
	\begin{itemize}
		\item[(a)] There is an isomorphism $\mathcal{C}(\Hy P^3)\to \mathcal{C}(\Hy P^2)$.
		\item[(b)] There is an isomorphism $\Theta_{16}\cong \mathbb{Z}_2\to \CC(\Hy P^4)$.
	   \item[(c)]  There is a split short exact sequence $$0 \to \Theta_{20} \to \mathcal{C}(\mathbb{H}P^5)\to \mathcal{C}(\mathbb{H}P^4)\to 0,$$ where $\Theta_{20}\cong \mathbb{Z}_{24}$.
	\end{itemize}
\end{thmb}
(cf. Theorem \ref{main}).

We also relate the group $\mathcal{C}(\Hy P^n)$ to the tangential smooth structures set of $\Hy P^n$, so that these computations also imply results for tangential smooth structures on $\Hy P^n$ for $n\leq 5$. As a consequence of this, we show that if $M$ is a smooth manifold tangential homotopy equivalent to $\mathbb{H} P^4$, then there is a homotopy $16$-sphere $\Sigma$ such that $M$ is diffeomorphic to $\mathbb{H} P^4\#\Sigma$.

Computations in the concordance group allows us to discuss free smooth actions of the unit quaternionic sphere on homotopy spheres. The study of such actions has been of considerable interest \cite{Hsi66, Ka12, Ku69, KK70, KK71}. The orbit spaces of such actions are always homotopy equivalent to $\Hy P^n$. We explore such actions such that the orbit spaces are homeomorphic to $\Hy P^n$, and prove that for $n=2,4$, there are exactly two non-equivalent actions on the standard sphere , while for $n=3$, there is an exotic $15$-sphere which supports such an action. 

\subsection{Organisation} In section \ref{iner-quart}, we introduce some preliminaries on the inertia group and make computations for $\Hy P^n$. In section \ref{smstr-quart}, we compute the set of concordance classes of $\Hy P^n$. Section \ref{s3act} deals with applications to $3$-sphere actions on homotopy spheres and Section \ref{sec5} deals with applications to tangential homotopy structures. 

\subsection{Notation} Denote by $O=\underset{n\to \infty}{\colim}~O_n$, $Top= \underset{n\to \infty}{\colim}~ Top_n$, $F=\underset{n\to \infty}{\colim}~ F_n$ and 
$SF=\underset{n\to \infty}{\colim}~ SF_n$ (\cite{KL66, LR65}), the direct limit of the groups of orthogonal transformations, self-homeomorphisms of $\mathbb{R}^n$ preserving the origin, base point preserving self-homotopy equivalences of $\mathbb{S}^{n}$, and base point preserving self-homotopy equivalences of degree one of $\mathbb{S}^{n}$ respectively. Let $F/O$ be the homotopy fibre of the canonical map $BO\to BF$ between the classifying spaces for stable vector bundles and stable spherical fibrations (see  \cite[\S 2, \S 3]{MM79}) and $Top/O$ be the homotopy fibre of the canonical map $BO\to BTop$ between the classifying spaces for stable vector bundles and stable topological $\mathbb{R}^n$-bundles (see \cite[Theorem 10.1 Essay IV]{KS77}). In this paper all manifolds will be oriented and connected, and all homeomorphisms and diffeomorphisms are assumed to preserve orientation, unless otherwise stated. We use the notation $\mathbb{S}^m$ for the standard unit $m$-sphere in $\R^{m+1}$, and while making computations in homotopy classes, the notation $S^m$ for a space homotopy equivalent to $\mathbb{S}^m$. 

\subsection{Acknowledgements} The research of the first author was partially supported by NBHM project ref. no. 2/48(11)/2015/NBHM(R.P.)/R\&D II/3743.

 \section{Inertia groups of quaternionic projective spaces}\label{iner-quart}
In this section, we make some computations in the inertia groups of $\Hy P^n$. We show that the concordance inertia group is trivial when $n=5$, but non-trivial in many cases in high dimensions. We begin by recalling some preliminaries on exotic spheres and inertia groups. 

\begin{definition}
		\begin{itemize}
			\item[(a)] A homotopy $m$-sphere $\Sigma^m$ is a closed smooth manifold homotopy equivalent to  $\mathbb{S}^m$\footnote{By the $h$-cobordism theorem, such a sphere is also homeomorphic to the standard sphere $\mathbb{S}^m$.}.
			\item[(b)]A homotopy $m$-sphere $\Sigma^m$ is said to be exotic if it is not diffeomorphic to $\mathbb{S}^m$.
			\item[(c)] Two homotopy $m$-spheres $\Sigma^{m}_{1}$ and $\Sigma^{m}_{2}$ are said to be equivalent if there exists a diffeomorphism $f:\Sigma^{m}_{1}\to \Sigma^{m}_{2}$.
		\end{itemize}
\end{definition}

		The set of  diffeomorphism classes of homotopy $m$-spheres is denoted by $\Theta_m$. The class of $\Sigma^m$ is denoted by [$\Sigma^m$]. When $m\geq 5$, $\Theta_m$ forms an abelian group with group operation given by connected sum $\#$, and the zero element represented by the equivalence class  of $\mathbb{S}^m$. M. Kervaire and J. Milnor \cite{KM63}\footnote{Kervaire and Milnor used the relation $h$-cobordism which agrees with orientation preserving diffeomorphism for $m\geq 5$.} showed that each  $\Theta_m$ is a finite abelian group; in particular,  the following calculations will be needed : $\Theta_{m}\cong \mathbb{Z}_2$, where $m=8, 16$, $\Theta_{12}=0$ and  $\Theta_{20}\cong \mathbb{Z}_{24}$.

\begin{definition}
	Let $M$ be a smooth manifold. Let $(N,f)$ be a pair consisting of a smooth manifold $N$ together with a homeomorphism $f : N\to M$. Two such pairs $(N_1, f_1)$ and $(N_2, f_2)$ are concordant if there exists a diffeomorphism $g : N_1\to N_2$ such that the composition $f_2\circ g$ is topologically concordant to $f_1$, i.e., there exists a homeomorphism $F : N_1\times [0, 1]\to M\times [0, 1]$ such that $F_{|N_1\times 0}=f_1$ and $F_{|N_1\times 1}=f_2\circ g$. The set of all such concordance classes is denoted by $\mathcal{C}(M)$. Recall that there is a canonical homeomorphism $h_{can}:M^m\#\Sigma^m \to M^m$,  which induces a class $[(M^m\#\Sigma^m,h_{can})]$ in $\mathcal{C}(M)$. Note that $[(M^m\#\mathbb{S}^m,h_{can})]$ is the class of $(M^m, Id)$.
\end{definition}

\begin{definition} 
		Let $M^m$ be a closed smooth $m$-dimensional manifold. Let $I(M)$ be the subgroup of $\Theta_{m}$ consisting of all $\Sigma \in \Theta_{m}$ such that $M\#\Sigma$ is diffeomorphic to $M$.\\
Let $I_h(M)$ be the subgroup of $\Theta_{m}$ consisting of all $\Sigma \in \Theta_{m}$ such that
there exists a diffeomorphism $M\#\Sigma \to M$ which is homotopic to the canonical homeomorphism $h_{can}:M\#\Sigma \to M$. Let $I_c(M)$ be the subgroup of $\Theta_{m}$ consisting of all $\Sigma \in \Theta_{m}$ such that $(M\#\Sigma, h_{can})$ is concordant to $(M,Id)$.
\end{definition}
Note that $I_c(M)\subset I_h(M) \subset I(M)$, and these might not be equal \cite[Theorem 2.1]{Sch87}.  For $\Sp^n$ or an exotic sphere $\Sigma^n$, all these groups are trivial. For the space $\C P^n$, these groups are all equal \cite{Kas15}. 

The inertia groups  are not homotopy invariant. One has $I(\Sp^3 \times \Sigma^{10}) \neq I(\Sp^3 \times \Sp^{10})$ from \cite[Corollary 2, 3]{Kaw69}. 
On the other hand, the group $I_c(M)$ is indeed homotopy invariant, and we recall the formulation below. 

We note from  \cite[p. 25 and 194]{KS77} that, for $m\neq 4$, $\CC (\Sp^m) \cong \Theta_m \cong [S^m, Top/O]$, and $\CC(M^m) \cong [M^m, Top/O]$.  Let $f_{M}:M^m\to S^m $  be a degree one map (which is well-defined up to homotopy). Composition with $f_{M}$ defines a homomorphism
$$f_{M}^*:[S^m, Top/O]\to [M^m ,Top/O],$$ 
and in terms of the identifications above
, $f_{M}^*$ becomes $[\Sigma^m]\mapsto [M^m\#\Sigma^m]$.  Therefore, the concordance inertia group $I_c(M)$ can be identified with $\Ker(f_M^\ast)$. 

In this paper, we are interested in the inertia groups of $\Hy P^n$. One notes from \cite[Corollary 3.2]{Kas17} that $I_h(\Hy P^n) = I_c(\Hy P^n)$ for $n\geq 2$. On the other hand, the inertia group of $\Hy P^n$ may be different from this. We use computations of $I_c(\Hy P^n)$ to deduce the results of the paper, and we note that this also implies results for the other inertia groups. For the concordance inertia group, we use the homotopy-theoretic description above for $M=\Hy P^{n}$. In \cite[Corollary 3.4.]{AF04}, it was proved that $f_{\Hy P^{n}}^*:[\mathbb{S}^{4n}, Top/O]\to [\Hy P^{n}, Top/O]$ is monic for $n = 2$ and $n = 4$. Since $\Theta_{12}=0$, so the case $n = 3$ is obvious. Hence, the concordance inertia group $I_c(\Hy P^{n})$ of  $\Hy P^{n}$ is trivial for $n\leq 4$; for $n=5$, the concordance inertia group is shown to have $2$-primary component of order at most $2$.  In contrary to the concordance and homotopy inertia groups, the inertia group of $\Hy P^{n}$ has a non-trivial element even for $n=2$. In \cite[Theorem A]{KS07}, it was proved that the inertia group of $\Hy P^{2}$ is isomorphic to the group of homotopy $8$-spheres $\Theta_8\cong \Z_2$. 

In this paper, we show that the inertia group of $\Hy P^{n}$ is non-trivial in many cases. For this, we prove the non triviality of the concordance inertia group of  $\Hy P^{n}$ by using the computations of the stable homotopy group of spheres. The first unresolved case is $n=5$ where we prove that the concordance inertia group is trivial. 
\begin{theorem}\label{iner}
The homomorphism $f_{\Hy P^{5}}^*:[S^{20}, Top/O]\to [\Hy P^{5}, Top/O]$  is monic.
	\end{theorem}
	\begin{proof}
We proceed as in \cite{AF04} by considering the map $Top/O \to F/O$ and using the fact $F/O_{(p)} \simeq BSO_{(p)} \times Cok J_{(p)}$ which is a $p$-local splitting of infinite loop spaces  from \cite[Theorem 5.18]{MM79}. Recall that $Cok(J)$ is defined in \cite[Definition 5.16]{MM79} as the fibre of a map $F/O \to BSO$. It is an infinite loop space whose homotopy groups are isomorphic to the cokernel of the $J$-homomorphism, that is, $\pi_\ast (S^0)/Im(J)$. The splitting is as H-spaces if $p$ is an odd prime. So, we have a commutative diagram of Abelian groups
$$\xymatrix{  [\Hy P^5, Top/O]_{(p)} \ar[r]^{\alpha'} & [\Hy P^5, F/O]_{(p)}  & [\Hy P^5, Cok J_{(p)}] \ar[l]_{\beta'} \\ 
 [S^{20}, Top/O]_{(p)} \ar[r]^\alpha \ar[u]^{f^\ast_{\Hy P^5}} & [S^{20}, F/O]_{(p)} \ar[u]  & [S^{20}, Cok J_{(p)}]. \ar[l]_\beta \ar[u]^{f^\ast_{\Hy P^5}} \\}$$
In the above diagram we are also using $[-,Cok(J)_{(p)}] \cong [-,Cok(J)]_{(p)}$ which is true because $Cok(J)$ is an infinite loop space. We know from the results of Kervaire and Milnor (\cite{KM63}) that $\alpha$ is injective as $bP_{n+1}=0$ for $n$ even, and has the same image as $\beta$ which is the cokernel of the $J$-homomorphism. On account of the splitting $F/O_{(p)} \simeq BSO_{(p)} \times Cok(J)_{(p)}$ \cite[Theorem 5.18]{MM79}, the map $\beta'$ is injective.  Therefore, in order to prove that $[S^{20}, Top/O]\to [\Hy P^{5}, Top/O]$ is monic, it suffices to prove that $[S^{20}, Cok J_{(p)}]\to [\Hy P^{5}, Cok J_{(p)}]$ is monic for every prime $p$. As $\pi_{20}^s \cong \Z_{24}$, the primes we need to consider are $2$ and $3$. 

We start with the prime $2$. Consider the diagram 
$$\xymatrix{  \{S^{17}, Cok J_{(2)}\} \ar[d]^{q^\ast} \\ 
 \{\Sigma \Hy P^4, Cok J_{(2)}\} \ar[r]  & \{S^{20}, Cok J_{(2)}\} \ar[r]  & \{\Hy P^5, Cok J_{(2)}\} }$$
The bottom row is part of the long exact sequence for the cofibre $S^{19} \to \Hy P^4 \to \Hy P^5$, and the vertical map is induced by $q: \Hy P^4 \to \Hy P^4/ \Hy P^3 \simeq S^{16}$.  In \cite[Lemma 3.2]{AF04}, it was checked that $q^\ast$ is surjective. The composite 
$$ \{S^{17}, Cok J_{(2)}\} \stackrel{q^\ast}{\to}   \{\Sigma \Hy P^4, Cok J_{(2)}\} \to  \{S^{20}, Cok J_{(2)}\} $$
is induced by the map $\alpha : S^{20} \to S^{17}$ whose cofibre is the space $\Sigma (\Hy P^5/\Hy P^3)$. We note that the mod $2$ cohomology of $\Hy P^n$ is given by the ring $\Z/2[y]/(y^{n+1})$ with $|y|=4$, and it is easily computed that $Sq^4(y^4)=0$. The group $\pi_{20}(S^{17})_{(2)}$ is isomorphic to the stable homotopy group $\pi_3(S^0)_{(2)}\cong \Z/8$, which is generated by the class $\nu$. The class $\nu$ is detected by the cohomology operation $Sq^4$, which means that the operation $Sq^4$ is non-trivial on the mapping cone. It follows that $\alpha = 2\alpha'$ in the group $\pi_{20}(S^{17}_{(2)}) $. As $ 2\{S^{17}, Cok J_{(2)}\} =0 $, we deduce that $\alpha^\ast :  \{S^{17}, Cok J_{(2)}\}  \to  \{S^{20}, Cok J_{(2)}\}$ is $0$. Therefore, the kernel of $ \{S^{20}, Cok J_{(2)}\} \to  \{\Hy P^5, Cok J_{(2)}\} $ is trivial.     

Next we consider the prime $3$. We have the map $F_{(3)} \to Cok J_{(3)}$, and $Cok J_{(3)}$ is a summand of $F_{(3)}$. Therefore, it suffices to show $f^\ast_{\Hy P^5}:  [S^{20}, F_{(3)}] \to [\Hy P^5, F_{(3)}]$ is injective on the classes which come from cokernel of $J$. Finally, these homotopy classes may be computed using a single path component of $F$ and thus these may be computed using $\Omega^\infty S^0$, the infinite loop space associated to the sphere spectrum. Thus it suffices to compute $f^\ast_{\Hy P^5}: {\pi_{20}^s}\otimes \Z_{(3)}= \{ S^{20}, S^0_{(3)} \} \to \{ \Hy P^5, S^0_{(3)} \}$ as all the classes in ${\pi_{20}^s} \otimes \Z_{(3)}$ lie in  the cokernel of the $J$-homomorphism (observe that $20$ is not congruent to $-1 \pmod{4}$ and so the $3$-local $J$-homomorphism is $0$ in this degree). As in the case above, we compute $\{\Sigma \Hy P^4, S^0_{(3)}\}$ and prove that it is $0$ using computations in \cite[Table A.3.4]{Rav04}. 

We have $\{\Sigma \Hy P^1, S^0_{(3)}\} = {\pi_5^s} \otimes \Z_{(3)} = 0$, and therefore, the exact sequence 
$$ \{S^9, S^0_{(3)}\} \to \{\Sigma \Hy P^2, S^0_{(3)}\} \to  \{\Sigma \Hy P^1, S^0_{(3)}\} $$
together with the fact ${\pi_9^s} \otimes \Z_{(3)} = 0$  implies that $\{\Sigma \Hy P^2, S^0_{(3)}\} =0$. Next we have the exact sequence 
$$ \{\Sigma^2 \Hy P^2, S^0_{(3)}\}\to  \{S^{13}, S^0_{(3)}\} \to \{\Sigma \Hy P^3, S^0_{(3)}\} \to  \{\Sigma \Hy P^2, S^0_{(3)}\} $$
in which the right hand term is $0$. The group $\pi_{13}^s \otimes \Z_{(3)} \cong \Z/3\{\alpha_1\beta_1\}$ from \cite[Table A.3.4]{Rav04}. For computing the term $ \{\Sigma^2 \Hy P^2, S^0_{(3)}\}$ we note $ \{\Sigma^2 \Hy P^1, S^0_{(3)}\} = \pi_6^s \otimes \Z_{(3)} =0$ and thus $ \{\Sigma^2 \Hy P^2, S^0_{(3)}\} \cong \pi_{10}^s \otimes \Z_{(3)} \cong \Z/3\{\beta_1\}$. Thus we have the diagram
$$\xymatrix{  \{S^{10}, S^0_{(3)}\} \ar[d]^{q^\ast}_{\cong} \\ 
 \{\Sigma^2 \Hy P^2, S^0_{(p)}\} \ar[r]  & \{S^{13}, S^0_{(3)}\}  }$$
such that the composite computes the map in the exact sequence above. This is induced by the map $S^{13} \to S^{10}$ whose cofibre is $\Sigma^2(\Hy P^3 /\Hy P^1)$. We compute the mod $3$ cohomology of $\Hy P^n$ as  $\Z/3[y]/(y^{n+1})$ with $|y|=4$, and so for the Steenrod power operation $\PP^1$, $\PP^1(y^2)=y^3$. As the operation $\PP^1$ detects $\alpha_1$ in the stable homotopy groups of spheres, it follows that the map $S^{13} \to S^{10}$ is a non-trivial multiple of $\alpha_1$. Therefore $\{\Sigma^2 \Hy P^2, S^0_{(3)}\}\to  \{S^{13}, S^0_{(3)}\} $ takes $\beta_1$ to $\alpha_1\beta_1$ and is thus an isomorphism. Hence $ \{\Sigma \Hy P^3, S^0_{(3)}\}=0$. Finally from the exact sequence 
$$ \{S^{17}, S^0_{(3)}\} \to \{\Sigma \Hy P^4, S^0_{(3)}\} \to  \{\Sigma \Hy P^3, S^0_{(3)}\} $$
and the fact that $\pi_{17}^s\otimes \Z_{(3)}=0$, we deduce $\{\Sigma \Hy P^4, S^0_{(3)}\}=0$. This completes the proof of the theorem. 
\end{proof}
From the above result, we have the following:	
\begin{corollary}\label{exotic}
For any two elements $\Sigma_1, \Sigma_2\in \Theta_{20}$,  $\Hy P^{5}\#\Sigma_1$ is concordant to $\Hy P^{5}\#\Sigma_2$ if and only if $\Sigma_1=\Sigma_2$. In particular, the concordance inertia group $I_c(\Hy P^{5})=0$.
\end{corollary}

Computations such as Theorem \ref{iner} have geometric applications along the lines of \cite{AF04}. We may start with the quaternionic hyperbolic manifold $N$ of dimension $20$ given by \cite{Bor63} which has a finite-sheeted cover $M$ that have a non-zero tangential map to $\Hy P^5$ \cite[Theorem 5.1]{Oku02}. Now from Okun's stronger result \cite[Theorem 3.6]{Oku02} and Theorem \ref{iner}, we have the following. 
\begin{theorem} \label{exohyp}
There exist twelve exotic spheres $\{\Sigma_i : ~~i=1~{\rm to}~ 12\}\subset \Theta_{20}$ and a closed quaternionic hyperbolic manifold $M^{20}$ of quaternion dimension $5$ such that the following is true.
\begin{itemize}
\item[(i)] The manifolds $M^{20}$, $\{M\#\Sigma_i : ~~i=1~{\rm to}~ 12\}$ are pairwise non-diffeomorphic.
\item[(ii)] Each of the manifolds $M\#\Sigma_i $ supports a Riemannian metric whose sectional curvatures are all negative.
\end{itemize}
For every closed quaternionic hyperbolic manifold $N$, there is a finite sheeted cover which satisfies the conclusions for $M$ above. 
\end{theorem}

We have thus observed that for $n\leq 5$, $I_c(\Hy P^5)=0$.  However, this is a phenomenon only in low dimensions, as it is possible to construct a fairly large number of non-trivial elements in the inertia groups of high dimensional quaternionic projective spaces. The technique for constructing these has been used in \cite[Theorem 3.9]{BK} for complex projective spaces in dimensions $4n+2$. 

We use the result from \cite{CNL96} :  For $p\geq 7$ the classes $\alpha_1\beta_1^r\gamma_t$ are non trivial in the stable homotopy groups of $S^0$ for $2\leq t \leq p-1$ and $r\leq p-2$ (in dimension $n(t,p,r)= [2(tp^3 - t - p^2) +2r(p^2 - 1 - p) -2]$).  With these assumptions $\beta_1^r\gamma_t$ is also non-trivial in dimension $n(t,p,r)-(2p -3)$. Note that whenever $r$ is even, $4 \mid n(t,p,r)$, and if $p\nmid t+r$, $p\nmid n(t,p,r)-2(p-1)$.  

\begin{theorem}\label{inerhigh}
Suppose that $p$ is a prime $\geq 7$, $2\leq t \leq p-1$ and $r\leq p-2$. Assume that $r$ is even and $p$ does not divide $t+r$. Under these assumptions, the map   
$$[S^{n(t,p,r)}, Top/O] \to [\Hy P^{\frac{n(t,p,r)}{4}}, Top/O]$$
induced by the degree one map $\Hy P^{\frac{n(t,p,r)}{4}} \to S^{n(t,p,r)}$ has non-trivial $p$-torsion in the kernel. 
\end{theorem}

\begin{proof}
We note that in $H^\ast (\Hy P^\infty;\Z/p)\cong \Z/p[y]$, the Steenrod operation $\PP^1(y^k)\neq 0$ whenever $p\nmid k$. Once we assume the stated hypothesis, it follows that $\PP^1(y^{\frac{n(t,p,r)}{4} - \frac{p-1}{2}}) \neq 0$. 

Let $N=\frac{n(t,p,r)}{4}$ and $M=\frac{n(t,p,r)}{4} - \frac{p-1}{2}$. We consider the map $q:\Hy P^N \to S^{4N}$ which quotients out the $(N-1)$-skeleton. This is the usual degree one map from $\Hy P^N$ to $S^{4N}$ up to homotopy. We deduce that for the map $q^\ast : [S^{4N}, Cok J_{(p)}] \to [\Hy P^N, Cok J_{(p)}]$, $\Ker(q^\ast)$ has non-trivial $p$-torsion. From the proof of Theorem \ref{iner}, observe that it suffices to prove this. We start by noting that $q$ splits into a composite 
$$\Hy P^N \stackrel{q_1}{\to} \Hy P^N/\Hy P^{M-1} \stackrel{q_2}{\to} S^{4N}$$
so that it suffices to verify that $\Ker(q_2^\ast)$ has non-trivial $p$-torsion. The space $\Hy P^N/ \Hy P^{M-1}$ has a CW-complex structure described as $S^{4M}\cup e^{4M+4} \cup \cdots \cup e^{4N}$. Working $p$-locally and in the stable homotopy category we work out  the attaching maps which are of the form $S^{4M + 4k -1} \to  S^{4M}\cup e^{4M+4} \cup \cdots \cup e^{4M+4(k-1)}$ for $1\leq k \leq \frac{p-1}{2}$. We note that on stable homotopy classes one has long exact sequences for any $X$ and a cell $e^m$ attached to $X$, 
$$\cdots \{S^r, X \} \to \{ S^r, X\cup e^m\} \to \{ S^r, S^m \} \to \{ S^r, \Sigma X \} \cdots $$
We note that for $1\leq k \leq \frac{p-1}{2}$, $\{ S^{4M+4k -1}, S^{4M + 4s}\}_{(p)}$ equals $0$, unless $k=\frac{p-1}{2}$ and $s=0$. Therefore  the map $S^{4M + 4k -1} \to  S^{4M}\cup e^{4M+4} \cup \cdots \cup e^{4M+4(k-1)}$ is homotopically trivial unless $k=\frac{p-1}{2}$. For the attaching map $S^{4N-1}$ the map goes down to the sphere $S^{4M}$, that is it is homotopic to a composite $S^{4N-1} \to S^{4M} \to  S^{4M}\cup e^{4M+4} \cup \cdots \cup e^{4N-4}$. Therefore, it follows that 
$$\Hy P^N/ \Hy P^{M-1}_{(p)} \simeq (S^{4M} \cup e^{4N})_{(p)} \vee S^{4M+4}_{(p)} \vee \cdots \vee S^{4N-4}_{(p)}$$
We denote $X= (S^{4M}\cup e^{4N})_{(p)}$ in the above, so that it suffices to prove that the kernel of $q_3^\ast$ has non-trivial $p$-torsion where $q_3:X \to S^{4N}$. The attaching map $S^{4N-1} \to S^{4M}$ must be a non-trivial multiple of $\alpha_1$ as the operation $\PP^1$ carries the cohomology generator in degree $4M$ to the generator in degree $4N$, as $p \nmid M$. This is where we use $\PP^1(y^{\frac{n(t,p,r)}{4} - \frac{p-1}{2}}) \neq 0$.  Thus $X\simeq S^{4M} \cup_{\alpha_1} S^{4N}$. Now we have the long exact sequence 
$$\cdots  \to [S^{4M+1}, Cok J_{(p)}] \to  [S^{4N}, Cok J_{(p)}] \to  [X, Cok J_{(p)}] \cdots $$
where the left map is induced by multiplication by $\alpha_1$. We proceed as in the proof of the $p=3$ case of  Theorem \ref{iner}, and use that $Im J_{(p)}$ is trivial in degrees which are $0, 1 \pmod 4$. Thus it suffices to check that the left arrow hits a non-trivial element when $Cok J_{(p)}$ is replaced by the sphere spectrum $ S^0_{(p)}$. Here, we know that $\beta_1^r \gamma_t$ is carried to $\alpha_1 \beta_1^r \gamma_t$ by the discussion preceeding the Theorem, and thus we are done. 
\end{proof}

\begin{remark} \rm{
Theorem \ref{inerhigh} shows that under the given hypothesis, the concordance inertia group $I_c(\Hy P^\frac{n(t,p,r)}{4})$ has non-trivial $p$-torsion. Observe that the hypothesis may be easily satisfied. The first example arises when $t=2,r=2$ which says that $I_c(\Hy P^{p^3 -p + \frac{p^2-5}{2}})$ has non-trivial $p$-torsion. Specializing further to $p=7$ we get that $I_c(\Hy P^{358})$ has non-trivial $7$-torsion. It follows that in these cases the homotopy inertia group and the inertia groups are also non-trivial. 
}
\end{remark}

\section{Smooth Structures on Quaternionic Projective Spaces}\label{smstr-quart}
 Computations of the group $\CC(M)$ are very important in the study of classification of manifolds and smooth structures, and from the identification $\CC(M) \cong [M,Top/O]$, this may be computed using homotopy theory.  In this section, we study $\CC(\Hy P^n)$ for low values of $n$. 

In the case $n=2$, it was proved in \cite[Theorem 2.7]{Kas16} that 
$$\CC (\Hy P^2)\cong \{[(\Hy P^2\#\Sigma^8,h_{can})], [(\Hy P^2,Id)] \} \cong\Z_2,$$ 
where $\Sigma^8\in \Theta_8$ is the exotic $8$-sphere.  In the following, we are interested in  the group $\CC (\Hy P^n)$ for $n\geq 3$.

For the computations below, we  use the cofiber sequence  
$$S^{4n-1} \stackrel{p}{\to} \Hy {P}^{n-1}\stackrel{i} {\to} \Hy P^{n} \stackrel{f_{\Hy P^{n}}} {\to} S^{4n}$$
which induces the long exact sequence
\begin{equation}\label{longG}
\cdots \to [\Sigma \Hy P^{n-1}, Top/O] \to [S^{4n}, Top/O] \stackrel{f^{*}_{\Hy P^{n}}}{\to}[\Hy P^{n}, Top/O] \stackrel{i^{*}}{\to}[\Hy P^{n-1}, Top/O] \cdots .
\end{equation}
Using these techniques, we compute
 \begin{theorem}\label{main} 
 \begin{itemize}
\item[(i)] $\CC(\Hy P^3)$ has two concordance classes and therefore is $\cong \Z_2$. These classes induce the two different concordance classes on $\Hy P^2$ via $i^\ast$ in \eqref{longG}. \\
\item[(ii)] $\CC(\Hy P^4)$ has exactly two concordance classes. The two classes are obtained as  $$\left \{ [(\Hy P^4\#\Sigma,h_{can})]~ |~\Sigma\in \Theta_{16} \right \},$$ where $h_{can}:\Hy P^4\#\Sigma \to \Hy P^4$ is the canonical homeomorphism.
\item[(iii)]  There is a split short exact sequence $$0 \to \Theta_{20} \to \mathcal{C}(\mathbb{H}P^5)\to \mathcal{C}(\mathbb{H}P^4)\to 0,$$ where $\Theta_{20}\cong \mathbb{Z}_{24}$.
\end{itemize}
  \end{theorem}
  \begin{proof}
 We start by proving (i). In the exact sequence \eqref{longG}, we use the fact  $[S^{12}, Top/O]\cong \Theta_{12}=0$, to deduce that  
$$i^{*}:[\Hy  P^{3}, Top/O] {\to} [\Hy  P^{2}, Top/O]$$ 
is a monomorphism.  Since $f^{*}_{\Hy  P^{2}} : [S^{8}, Top/O] \to [\Hy P^{2}, Top/O]$ is an isomorphism and  $[S^{8}, Top/O]\cong \Theta_8\cong  \Z_2$, the non-trivial element in $[\Hy P^{2}, Top/O]$ is represented by a map 
$$g:\Hy P^{2} \stackrel{f_{\Hy P^{2}}}{\to} S^{8} \stackrel{\Sigma}{\to} Top/O,$$
 where $\Sigma:  S^{8}\to Top/O$ represents the exotic $8$-sphere in $\Theta_8$. Therefore, the effect of $p^{*}: [\Hy  P^{2}, Top/O]{\to} [S^{11},Top/O]\cong \Z_{992}$ on the homotopy class $[g]$ is  represented by the map 
$$S^{11}\stackrel{p}{\to} \Hy P^{2}\stackrel{f_{\Hy P^{2}}}{\to} S^{8} \stackrel{\Sigma}{\to} Top/O.$$ 
Now we will use the fact that the composition $f_{\Hy P^{2}}\circ p:S^{11}\to  S^{8}$ is multiplication by $2\nu_2$ \cite[page 38]{Jam76}, where  $\nu_2 = \Sigma^4 \nu \in \pi_{11}(S^8)$ is the $4$-fold suspension of the Hopf map $\nu$. As $2 [\Sigma]=0$, we have 
$$p^\ast([g]) = (2\nu_2) ([\Sigma]) = \nu_2(2[\Sigma]) = 0$$
where the second equality is derived from the fact that $2$ commutes with $\nu$ in the stable range. It follows that  $p^{*}: [\Hy P^{2}, Top/O]{\to} [S^{11},Top/O]$ is the zero map. Therefore the map 
$$i^{*}:[\Hy P^{3}, Top/O]{\to}[\Hy P^{2}, Top/O]$$ 
is an isomorphism and hence $\CC(\Hy P^3)\cong \Z_2$.

 Now consider the case (ii), that is, $n=4$. We prove that the map $[\Hy P^3, Top/O] \to [S^{15}, Top/O]$ induced by the attaching map of $\Hy P^4$ is injective.  Since $[\Hy P^3, Top/O]\cong \Z/2$ we work $2$-locally. We use the commutative diagram 
\begin{equation} \label{conc}
\xymatrix{  [\Hy P^3, Top/O]_{(2)} \ar[r]^{\alpha} \ar[d] & [\Hy P^3, F/O]_{(2)} \ar[d] & [\Hy P^3, Cok J_{(2)}] \ar[l]_{\beta} \ar[d] \\ 
 [S^{15}, Top/O]_{(2)} \ar[r]^\alpha  & [S^{15}, F/O]_{(2)}   & [S^{15}, Cok J_{(2)}]. \ar[l]_\beta } 
 \end{equation}
We have the splitting $F/O_{(p)} \simeq BSO_{(p)} \times Cok(J)_{(p)}$ \cite[Theorem 5.18]{MM79}, so that the maps denoted by $\beta$ are injective. We observe that in the top row, the map $\alpha$ is  injective and has the same image as $\beta$. For this, consider the commutative diagram 
$$\xymatrix{  [\Hy P^3, Top/O]_{(2)} \ar[r]^{\alpha} \ar[d]^{\cong} & [\Hy P^3, F/O]_{(2)} \ar[d] & [\Hy P^3, Cok J_{(2)}] \ar[l]_{\beta} \ar[d]^{\cong} \\ 
 [\Hy P^2, Top/O]_{(2)} \ar[r]^\alpha  & [\Hy P^2, F/O]_{(2)}   & [\Hy P^2, Cok J_{(2)}]. \ar[l]_\beta \\
  [S^8, Top/O]_{(2)} \ar[r]^{\alpha} \ar[u]^{\cong} & [S^8, F/O]_{(2)} \ar[u] & [S^8, Cok J_{(2)}]. \ar[l]_{\beta} \ar[u]^{\cong} \\  } $$
The left vertical isomorphisms follow from the discussion above. The right vertical isomorphisms follow from the fact that $\pi_k Cok(J)_{(2)}=0$ if $k=4,5, 11, 12$  \cite[Table A.3.3]{Rav04}. The $\alpha$ in the bottom row is injective and image of $\alpha$ equals the image of $\beta$ as in the proof of Theorem \ref{iner}. As $\pi_{12}(F/O)_{(2)} \cong \pi_{12} BSO_{(2)}$ which is torsion-free, the kernel of the middle vertical map is torsion-free and thus intersects trivially with the images of $\alpha$ and $\beta$. This completes the intended observation, so using \eqref{conc} it suffices to prove injectivity of $[\Hy P^3, Cok(J)_{(2)}] \to [S^{15},Cok(J)_{(2)}]$. We note that $\pi_4 (Cok J_{(2)})=0$. Therefore, it suffices to compute the kernel of the map
$$q^\ast : [\Hy P^3/\Hy P^1, Cok(J)_{(2)}] \to [S^{15}, Cok(J)_{(2)}]$$
with $q$ induced by $S^{15} \to \Hy P^3 \to \Hy P^3/\Hy P^1$. We use stable homotopy theory to complete the computation -- we assume that the spaces used below are actually suspension spectra and we use the notation of stable stems from \cite{Rav04}. We are allowed to do this because the space $Cok(J)_{(2)}$ is an infinite loop space. 

The space $\Hy P^3/ \Hy P^1$ is a cell complex with $2$-cells of dimensions $8$ and $12$. As $\pi_{12} Cok(J)_{(2)} = 0$ and $\pi_{11}Cok(J)_{(2)}=0$ \cite[Table A.3.3]{Rav04} we have that the group $[\Hy P^3/\Hy P^1, Cok(J)_{(2)}]$  is $\Z/2$ generated by the class $c_0$. From \cite[Page 38]{Jam76}, we note that the attaching map of the $12$-cell onto the $8$-cell is $2h_2$ and, we also note that the composite 
$$S^{15} \to \Hy P^3/\Hy P^1 \to \Hy P^3/\Hy P^2 \simeq S^{12}$$
is $3h_2$. Therefore, the map $S^{15} \to \Hy P^3/\Hy P^1$ (in the category of spectra) is described by a map $\phi: S^{15} \to S^{12}$ which is $\simeq 3h_2$, together with a specific choice of null-homotopy of $2h_2 \circ \phi$. This means that the pullback  $q^\ast c_0$ is an element of the Toda bracket $\langle 3h_2,2h_2, c_0\rangle$. Note that the indeterminacy of the bracket is $\pi_7. c_0 + 3h_2 . \pi_{12}$ which is  $0$ as  $\pi_{12} =0$, and the product of $c_0$ with $h_3$ is $0$. 

Thus it remains to compute the Toda bracket $\langle 3h_2, 2h_2, c_0 \rangle$ which is an odd multiple of $\langle h_2, h_0h_2, c_0 \rangle$. Now we can compute using relations described in \cite{Ko96}. We note the generator of $\pi_{15}(Cok(J)_{(2)})$ is $h_1 d_0$. Now we have the following two Toda brackets  : $d_0= \langle h_0,h_1,h_2,c_0 \rangle$ \cite[Page 250]{Ko96} and $h_0h_2 = \langle h_1,h_0,h_1 \rangle$ \cite[Page 179]{MT68}. 
We can now make some manoevres with Toda brackets
$$h_1d_0= h_1\langle h_0,h_1,h_2,c_0 \rangle  \subset \langle \langle h_1,h_0,h_1\rangle,h_2,c_0 \rangle$$ 
by \cite[Proposition 5.7.4 c]{Ko96}. Thus the above computations imply
$$ \langle \langle h_1,h_0,h_1\rangle,h_2,c_0\rangle  = \langle h_0h_2,h_2,c_0\rangle  \subset \langle h_2, h_0h_2, c_0 \rangle$$ 
where the last equation is implied by \cite[Proposition 5.7.4 b]{Ko96}. Therefore, $h_1d_0$ is an element of $\langle h_2, h_0h_2, c_0 \rangle$. Now, as the indeterminacy is trivial, this must be the same element as the image of $c_0$ in $[\Hy P^3, Cok(J)_{(2)}]$ which implies the required injectivity. 

Since the maps $f^{*}_{\Hy P^{4}}:[S^{16}, Top/O]\to [\Hy P^{4}, Top/O]$ and $p^{*}: [\Hy P^{3}, Top/O]{\to} [S^{15},Top/O]$ are injective, it follows from the exact sequence \eqref{longG}, that the map $f^{*}_{\Hy P^{4}}:[S^{16}, Top/O]\to [\Hy P^{4}, Top/O]$ is an isomorphism. Therefore 
$$\CC(\Hy P^4)=\left \{ [(\Hy P^4\#\Sigma, h_{can})]~|~\Sigma\in \Theta_{16} \right \}\cong \Z_2.$$

Finally we consider $n=5$. Observe from Theorem \ref{iner} and (ii) that the maps 
$$f^{*}_{\Hy P^{5}}:[S^{20}, Top/O]\to [\Hy P^{5}, Top/O]$$ 
is injective and 
$$f^{*}_{\Hy  P^{4}}:[S^{16}, Top/O]\to [\Hy P^{4}, Top/O]$$ 
is an isomorphism. Now we use the fact that the composition $f_{\Hy  P^{4}}\circ p: S^{19}\to  S^{16}$ is multiplication by $4\nu_2$ \cite[Page 38]{Jam76}, to deduce that the homomorphism 
$$p^{*}: [\Hy P^{4}, Top/O]{\to} [S^{19},Top/O]$$ 
is trivial as in (i). Therefore, from the exact sequence \eqref{longG}, it follows that there is an exact sequence 
\begin{equation}\label{exact}
0\to [S^{20},Top/O]\to [\Hy P^{5}, Top/O]{\to [\Hy P^{4}, Top/O]}\to 0.
\end{equation}
Now we prove that this exact sequence splits.  As $[S^{20},Top/O] \cong \Z_{24}$ and $[\Hy P^{4}, Top/O]\cong \mathbb{Z}_2$ by (ii), it is enough to show that the exact sequence splits after localizing at the primes $2$ and $3$. It is clear that at the prime 3, the exact sequence (\ref{exact}) splits. Next we consider the prime $2$. 
Consider the following commutative diagram :
\begin{equation*}\label{digram31}
\begin{CD}
0@>>> [S^{20},Top/O]@>f^{*}_{\Hy P^{5}}>> [\Hy P^{5}, Top/O]@>i^{*}>>  [\Hy P^{4}, Top/O]@>>>0\\
@.          @|          @Aq_{1}^{*}AA       @Aq_{1}^{*}A\cong A\\
0@>>> [S^{20},Top/O]@>q_{2}^{*}>> [\Hy P^{5}/\Hy P^{3}, Top/O]@>i^{*}>>  [\Hy P^{4}/\Hy P^{3}, Top/O]@>>>0
\end{CD}
\end{equation*}
where $q_{1}:\Hy P^{m}\to \Hy P^{m}/\Hy P^{n}$ is the quotient map, $q_{2}:\Hy P^{m}/\Hy P^{n}\to \mathbb{S}^{4m}$ is the collapsing map and $i: \Hy P^{4}/\Hy P^{3}\to \Hy P^{5}/\Hy P^{3}$ is the map  induced by the inclusion $\Hy P^{4}\to \Hy P^{5}$, and the isomorphism $q_{1}^*:[\Hy P^{4}/\Hy P^{3}, Top/O]\to [\Hy P^{4}, Top/O]$ is given by (ii). From this diagram, it follows that, in order to prove the top sequence is split exact sequence at the prime 2, it suffices to show that the attaching map $\mathbb{S}^{19}_{(2)}\to \Hy P^{4}/\Hy P^{3}_{(2)}$ of $\Hy P^{5}/\Hy P^{3}_{(2)}$ is null-homotopic. Note that the attaching map $\mathbb{S}^{19}\to \Hy P^{4}/\Hy P^{3}$ factors in the diagram as below
\begin{equation}\label{fact}
\xymatrix{S^{19} \ar[r] \ar[rd]   & \Hy P^{4}/\Hy P^{3}\simeq S^{16}  \\ 
	& \C P^9/\C P^7\simeq S^{18}\vee S^{16} \ar[u]_{r} }
\end{equation}
where the map $S^{19} \to \C P^9/\C P^7$ is induced by the Hopf fibration $S^{19} \to \C P^9$, $\C P^9/\C P^7\simeq S^{18}\vee S^{16}$ follows from the fact that $Sq^2(x^8)=0$ where $x$ is the generator of the cohomology of $\C P^9$, and the map $r:\C P^9/\C P^7\to \Hy P^{4}/\Hy P^{3}$ is a retraction induced by the natural map $\C P^{9} \to \Hy P^4$. By \cite[Proposition 5.2]{Mos68}, it follows that the map $S^{19} \to S^{18}\vee S^{16}$ is given by $(\eta,0)$.  Therefore, from the diagram (\ref{fact}), the attaching map $\mathbb{S}^{19}_{(2)}\to S^{16}_{(2)}$ is null-homotopic and hence the exact sequence (\ref{exact}) splits. This completes the proof.
 \end{proof}

\section{Smooth Free actions of $\Sp^3$ on Homotopy Spheres}\label{s3act}
The topic of the existence and classification of free actions of $\Sp^1$ and $\Sp^3$ on exotic spheres is of considerable interest. It is well known that these are the only compact connected Lie groups which have free differentiable actions on homotopy spheres \cite{Ku69, KK70}. In fact one may prove that $\Z_p\times \Z_p$ (see \cite{MTW76} for results about finite groups acting freely on spheres) does not act freely on a sphere, so that any Lie group acting freely on a sphere must have rank $1$.  

For spheres acting on spheres,  it follows from Gleason's lemma \cite{Gle50} that such an action is always a principal fibration which is homotopically equivalent to the classical Hopf fibration, i.e., it can be obtained from a pullback of the classical Hopf fibration by a homotopy equivalence (see \cite[Proposition 1]{Hsi66}).  However as smooth actions, there might be many inequivalent ones. Among many results in this regard, Hsiang \cite{Hsi66} has shown that in fact, there are always infinitely many differentiably inequivalent free actions of $\Sp^3$ (respectively, of $\Sp^1$) on $\Sigma^{4n+3}$ (respectively on $\Sigma^{2n+1}$) for $n\geq 2$ (respectively, for $n\geq 4$). Our interest lies in the case of $\Sp^3$-actions.
\begin{definition}\rm \cite[Definition 2.2]{Bur72}
Denote a differentiable action $T:G\times M\to M$ of a (Lie) group $G$ on a differentiable manifold $M$ by $(G, T, M)$. Two differentiable actions $(G, T_i, M_i)$, $i = 1, 2$, are called differentiably equivalent iff there exists an orientation preserving (Lie) group homomorphism $\alpha  : G\to G$, and a diffeomorphism $h : M_1\to M_2$ such that $T_2 (\alpha(g), h (x))= h(T_1 (g, x))$.
\end{definition}
Recall that given a differnetiable manifold $M^k$, sullivan \cite{Sul67} defines a homotopy smoothing of $M^k$ to be a homotopy equivalence of pairs, $f:(L^k,\partial L^k)\to (M^k,\partial M^k)$, where $L^k$ is a smooth manifold. Two homotopy smoothings $(L_{0}^k,f_0)$ and $(L_{1}^k,f_1)$ are called equivalent if there is a diffeomorphism $h:L_{0}^k\to L_{1}^k$ such that $f_2\circ h$ is homotopic to $f_1$. The set of equivalence classes of homotopy smoothings of $M^k$ is denoted by $\mathcal{S}^{\mathit Diff}(M^k)$. 
\begin{lemma} \cite[Proposition 2.6]{Bur72}\label{action}
There is a natural 1-1 correspondence between the equivalence classes of differentiable free $\mathbb{S}^3$-actions on homotopy $4n+3$-spheres and elements of $\mathcal{S}^{\mathit Diff}(\Hy P^n)$.
\end{lemma}
Recall that the correspondence given by Lemma \ref{action} is defined as follows. If $\Sp^3$ acts on $\Sigma^{4n+3}$, then, the principal $\Sp^3$-fibration $\Sigma^{4n+3}\to \Sigma^{4n+3}/\Sp^3$ is completely classified by its characteristic map, namely a homotopy class $f:\Sigma^{4n+3}/\Sp^3\to \Hy P^{\infty}$. Since $\Sigma^{4n+3}/\Sp^3$ has real dimension $4n$, the map  $f$ is factored through a map
$g:\Sigma^{4n+3}/\Sp^3\to \Hy P^{n}$, which is a homotopy equivalence (see \cite[Proposition 1]{Hsi66}). On the other hand if $P^{4n}$ is smooth and homotopy equivalent to $\Hy P^n$,  then there is a pullback
$$\xymatrix{M^{4n+3} \ar[d] \ar[r]^-{\tilde{h}} & \Sp^{4n+3} \ar[d]^{H}\\ 
	             P^{4n} \ar[r]^-{h} & \Hy P^n.}$$ 
where $H:\Sp^{4n+3} \to \Hy P^n$ is the Hopf bundle. It follows easily that the map $\tilde{h}$ is a homotopy equivalence so that $M$ is an exotic sphere $\Sigma^{4n+3}$.\\

Thus, by combining \cite[Theorem 3.1]{Kas17} with Lemma \ref{action}, we have the following :
\begin{theorem}\label{action-home}
There is a bijective correspondence between $[\Hy P^n, Top/O]$ and the set of all equivalence classes of differentiable free $\mathbb{S}^3$-actions on homotopy $4n+3$-spheres $\Sigma^{4n+3}$ for which the orbit space $\Sigma^{4n+3}/\mathbb{S}^3$ is homeomorphic to  $\Hy P^n$.
\end{theorem}
G. Brumfiel \cite{Bru68, Bru71}  has found all possible homotopy spheres in dimensions $9$, $11$ and $13$ which admit free differentiable actions of $\mathbb{S}^1$. He also studied the free differentiable actions of $\mathbb{S}^1$ on homotopy spheres which do not bound $\pi$-manifolds. Following Brumfiel, we study here all possible homotopy spheres which admit free differentiable actions of $\mathbb{S}^3$.

We start by recalling some facts from smoothing theory \cite{Bru68, Bru71}. Let $M^k$, $k\geq 6$, be a simply connected, closed combinatorial manifold with a differentiable structure in the complement of a point. Let $M^k_0=M^k\setminus {\rm int}(\mathbb{D}^k)$, where $\mathbb{D}^k$ is a topologically embedded disc. $M^k_0$ inherits a differentiable structure from $M^k\setminus \{p\}$, hence $\partial M^k_0$ belongs to $\Theta_{k-1}$.  In \cite{Sul67}, Sullivan constructs a bijection $\theta:\mathcal{S}^{Diff}(M_0)\stackrel{\cong} {\to}[M_0, F/O]$. Thus, if $h:M_{0}^{'}\to M_0$ represents an element of $\mathcal{S}^{Diff}(M_0)$, the formula $d\theta(M_{0}^{'}, h)=\partial M_{0}^{'}-\partial M_{0}$ defines a map $d: [M_0, F/O]\to \Theta_{k-1}$. Further, if $v\in {\rm Image}([M^k_0,Top/O]\stackrel{\psi_*}{\to}[M_0, F/O])$ then 
\begin{equation}\label{bou}
dv=\partial^*(v)\in \pi_{k-1}(Top/O)=\Theta_{k-1},
\end{equation}
where $\partial:S^{k-1}\to M^k_0$ represents the homotopy class of the inclusion of the boundary $\partial M_{0}\to M_0$. In particular, $d: [M_0, Top/O]\to \Theta_{k-1}$ is a group homomorphism (see \cite[pp-384]{Bru71}). Let $M=\Hy P^n$ and $\Hy P^{n+1}_0$ is regarded as the total space of the $\mathbb{D}^4$ bundle $H$ over $\Hy P^n$. Thus there are maps $$S^{4n+3}=\partial \Hy P^{n+1}_0 \stackrel{i} {\to}\Hy P^{n+1}_0\stackrel{H} {\to} \Hy P^n$$
 which induces 
$$[\Hy P^n, F/O]\stackrel{H^*}{\to}[\Hy P^{n+1}_0, F/O]\stackrel{\theta^{-1}} {\to}\mathcal{S}^{Diff}(\Hy P^{n+1}_0) \stackrel{i^*} {\to}\Theta_{4n+3},$$ 
where $i^*$ is the map which assigns to a homotopy smoothing of $\Hy P^{n+1}_0$  its boundary, which is a homotopy sphere. Note that $i^*\circ \theta^{-1}=d$. Denote by $\sigma$ the composite $\sigma=i^*\circ \theta^{-1}\circ H^*=d\circ H^*:[\Hy P^n, F/O]\to \Theta_{4n+3}$. \\

Since the map $\theta:\mathcal{S}^{Diff}(\Hy P^{n})\to [\Hy P^{n}, F/O]$ is an injective (\cite{Sul67}) and the proof of \cite[Proposition 1.1]{Bru68} works verbatim for $\Hy P^n$, we have the following result :
\begin{proposition}\label{tec1}
Let $f: P^{4n}\to \Hy P^n$ in $\mathcal{S}^{Diff}(\Hy P^{n})$ correspond to the $\mathbb{S}^3$ action on $\Sigma^{4n+3}$. Then $\Sigma^{4n+3}=\sigma\circ \theta (P^{4n},f)$.
\end{proposition}
\begin{remark}\label{actions}\rm
	\indent
\begin{itemize}
\item[(1.)] Note from Proposition \ref{tec1} that the image of the map $\sigma\circ \theta : \mathcal{S}^{Diff}(\Hy P^{n})\to \Theta_{4n+3}$ is identified with the set of all exotic spheres $\Sigma^{4n+3}$ which admit a differentiable free $\mathbb{S}^3$-action such that the orbit space $\Sigma^{4n+3}/\mathbb{S}^3$ is homotopy equivalent to $\Hy P^n$. 
\item[(2.)] From \ref{bou}, one can also see that the homomorphism $d: [\Hy P^n, Top/O]\stackrel{\psi_*}{\to}[\Hy P^n, F/O]\stackrel{\sigma}{\to}\Theta_{4n+3}=\pi_{4n+3}(Top/O)$ ($n\geq 5$) coincides with the map $[\Hy P^n, Top/O]\stackrel{p^*}{\to}\pi_{4n+3}(Top/O)$ induced by the Hopf map $p:S^{4n+3}\to \Hy P^n$. Now it follows from  Proposition \ref{tec1} that the image of $p^*:[\Hy P^n, Top/O]\to \Theta_{4n+3}$ consists of exotic spheres $\Sigma^{4n+3}$ which admit a differentiable free $\mathbb{S}^3$-action such that the orbit space $\Sigma^{4n+3}/\mathbb{S}^3$ is homeomorphic  to $\Hy P^n$. 
\end{itemize}
\end{remark}
 From the proof of Theorem \ref{main}((i),(ii),(iii)), we get.
\begin{theorem}\label{sphere}
	\begin{itemize}
		\item[(i)] $p^{*}: [\mathbb{H} P^{2}, Top/O]{\longrightarrow} [\mathbb{S}^{11},Top/O]$ is the zero map.
		\item[(ii)] $p^{*}: [\mathbb{H} P^{3}, Top/O]{\longrightarrow} [\mathbb{S}^{15},Top/O]$ is injective,
		 	\item[(iii)] $p^{*}: [\mathbb{H} P^{4}, Top/O]{\longrightarrow} [\mathbb{S}^{19},Top/O]$ is the zero map.
\end{itemize}	
\end{theorem}	
By combining Theorem \ref{main}(ii), the fact $[\mathbb{H} P^{2}, Top/O]\cong \mathbb{Z}_2$, and Theorem \ref{sphere} with Remark \ref{actions}(2), we can conclude the following results.
\begin{theorem}
For $n=2$ and $4$, there are exactly two equivalence classes of differentiable free $\mathbb{S}^3$-actions on $\mathbb{S}^{4n+3}$ such that the orbit space $\mathbb{S}^{4n+3}/\mathbb{S}^3$ is homeomorphic to  $\Hy P^n$.
\end{theorem}
\begin{theorem}
There exists an exotic $15$-sphere $\Sigma^{15}$ which does not bound a manifold such that  $\Sigma^{15}$ admits a free differentiable action of $\mathbb{S}^3$ such that the orbit space is homeomorphic to the quaternionic projective space $\mathbb{H} P^3$.
\end{theorem}
 \section{The Smooth Tangential Structures Set}\label{sec5}
 Recall from \cite[\S 2, \S 4]{MTW80} that there is a well-known tangential surgery exact sequence \begin{equation}\label{surgery}
 L_{m+1}(\mathbb{Z}\pi_1(M))\to \mathscr{S}^t(M) \stackrel{\eta^t}{\longrightarrow}[M, SF] {\longrightarrow}L_{m}(\mathbb{Z}\pi_1(M)).
 \end{equation}
 where $\mathscr{S}^t(M)$ is the tangential simple structure set for a closed smooth manifold $M$ of dimension $m\geq5$. In this paper, following Hertz \cite[pp.518 and Lemma 2.1]{Her69}, we use the following variant of the tangential smooth structure set.\\
Recall that two manifolds $N$ and $M$ are called tangentially homotopy equivalent if there is a homotopy equivalence $f: N\to M$ such that for some integers $k$, $l$, $f^{*}T(M)\oplus \displaystyle{\epsilon}_{N}^{k}\cong T(N)\oplus \displaystyle{\epsilon}_{N}^{l}$. 
 \begin{definition}\label{tang}\rm
 Let $M$ be a smooth manifold. Let $(N,f)$ be a pair consisting of a smooth manifold $N$ together with a tangential homotopy equivalence $f : N\to M$. Two such pairs $(N_1, f_1)$ and $(N_2, f_2)$ are equivalent if there exists a diffeomorphism $g : N_1\to N_2$ such that the composition $f_2\circ g$ is homotopic to $f_1$. The set of all such equivalence classes is denoted by $\mathcal{S}^{\mathit t}(M)$\footnote{Hertz denotes this by $\theta(M)$.}
\end{definition}
\begin{remark}\rm
Note that the tangential smooth structure set $\mathcal{S}^{\mathit t}(M)$ of $M$ is slightly different from the one described in the exact sequence \ref{surgery}. For example., from the surgery exact sequence \ref{surgery}, the structure set $\mathscr{S}^t(\mathbb{S}^7)$ is an infinite set , whereas the tangential smooth structure set $\mathcal{S}^{\mathit t}(\mathbb{S}^7)\cong \Theta_7$, which is a finite abelian group.
\end{remark} 
For $M=\Hy P^{n}$, Hertz \cite{Her69} gave an inductive geometric procedure where by representatives for all elements of $\mathcal{S}^{\mathit t}(\Hy P^{n})$ may be constructed from elements of $\mathcal{S}^{\mathit t}(\Hy P^{n-1})$ and showed that $\mathcal{S}^{\mathit t}(\Hy P^{2})$ contains atmost two elements, namely 
$$\mathcal{S}^{\mathit t}(\Hy P^2)=\{[(\Hy P^2\#\Sigma^8,h_{can})]~|~\Sigma^8\in \Theta_{8}=\mathbb{Z}_2\}.$$ Now it follows from \cite[Corollary 3.4]{Kas17} that  
$\mathcal{S}^{\mathit t}(\Hy P^2)$ contains exactly two elements.
In this section, we extend the computation to $n=3, 4, 5$. In the case $n=4$ the calculation also allows us to conclude a classification of diffeomorphism classes of smooth manifolds in the tangential homotopy type of $\Hy P^4$.\\

Let $M$ be a closed smooth manifold homotopy equivalent to $\Hy P^{n}$, $n\geq 2$. By the surgery exact sequence(\cite{Bro72}), we have 
\begin{equation}\label{ex1}
0\to\mathcal{S}^{Diff}(M)\to [M, F/O]\stackrel{\omega}{\longrightarrow}L_{4n}(\mathbb{Z}).
\end{equation}
Note that there are injective forgetful maps $F_t:\mathcal{S}^{\mathit t}(M)\to \mathcal{S}^{Diff}(M)$ and $F_c:\mathcal{C}(M)\to \mathcal{S}^{Diff}(M)$ (\cite[Theorem 3.1]{Kas17}). To study $[M, F/O]$ we use the exact sequence
$$\widetilde{KO}^{-1}(M)\to [M, SF]\stackrel{\phi_* }{\longrightarrow} [M, F/O]\to \widetilde{KO}^0(M)\to \bar{J}(M)$$
induced from fibrations 
$$SO\to SF\stackrel{\phi }{\longrightarrow}F/O\to BSO\to BSF.$$
Since $\widetilde{KO}^0(M)$ is free abelian group (\cite{SS73}), 
it follows that the image $${\rm Im}([M, SF]\stackrel{\phi_* }{\longrightarrow}  [M, F/O])$$ is the torsion subgroup of $[M, F/O]$. If the homotopy equivalence $f : N\to M$ represents an element of $\mathcal{S}^{Diff}(M)$, its image in $\widetilde{KO}^0(M)$ is given
by 
$$(f^{-1})^{*}T^0N-T^0M,$$ 
where $T^0(N)$ is the stable tangent bundle of $N$. Thus, a normal map $\varphi\in [M, F/O]$ represents an element $[N,f]\in \mathcal{S}^{\mathit t}(M)$  if and only if $$\varphi\in {\rm Im}([M, SF]\stackrel{\phi_* }{\longrightarrow}  [M, F/O])$$ and $\omega(\varphi)=0$. Therefore we have the following result: 
\begin{corollary}\label{sur}
Let $M$ be a closed smooth manifold homotopy equivalent to $\mathbb{H} P^{n}$, $n\geq 2$ and let $\omega :{\rm Im}([M, SF]\stackrel{\phi_* }{\longrightarrow}  [M, F/O])\to L_{4n}(\mathbb{Z})$ be the surgery obstruction. Then the image of $\mathcal{S}^{\mathit t}(M)\to [M,F/O]$ is ${\rm \Ker}(\omega)$. In particular, there is a bijection $\mathcal{S}^{\mathit t}(M)\cong {\rm \Ker}(\omega)$.
\end{corollary} 
 Now we prove the following.
 \begin{theorem}\label{imge}
Let $M$ be a closed smooth manifold homotopy equivalent to $\mathbb{H} P^{n}$, $n\geq 2$. Then 
 	\begin{center}
 		${\mathit Im}([M, SF]\stackrel{\phi_* }{\longrightarrow} [M, F/O])={\mathit Im}([M, Top/O]\stackrel{\psi_* }{\longrightarrow} [M, F/O]).$
 	\end{center}
 \end{theorem}
 \begin{proof}
 The proof follows from the argument given in the proof of Theorem 5.2 in \cite{Kas17} by noting that $\widetilde{KO}^0(M)$ is free abelian group, $H_k(M;\mathbb{Z})=0$ for all $k\not\equiv{\rm 0~ mod~ 4}$ and the map $\psi_*:[M,Top/O]\to [M, F/O]$ is an injective.
 \end{proof}
\begin{theorem}\label{tanset}
Let $M$ be a closed smooth manifold homotopy equivalent to $\mathbb{H}P^{n}$, $n\geq 2$. Then $${\rm Im}(\mathcal{S}^{\mathit t}(M)\stackrel{F_t}{\longrightarrow}\mathcal{S}^{Diff}(M))={\rm Im}(\mathcal{C}(M)\stackrel{F_c}{\longrightarrow}\mathcal{S}^{Diff}(M)).$$ In particular, there is a bijection $\mathcal{S}^{\mathit t}(M)\cong \mathcal{C}(M)$.
\end{theorem}
 \begin{proof}
 It follows from the exact sequence (\ref{ex1}) that a normal map $\varphi\in [M, F/O]$ represents an element $[N,f]\in \mathcal{S}^{Diff}(M)$  if and only if $\omega(\varphi)=0$. Observe that if $\varphi$ lies in the image $$[M,Top/O]\xhookrightarrow{F_c}\mathcal{S}^{Diff}(M)\xhookrightarrow{inj} [M,F/O],$$
 then $\omega(\varphi)=0$. Therefore the surgery obstruction 
$$\omega:{\mathit Im}([M,Top/O]\stackrel{\psi_*}{\longrightarrow}[M, F/O])\to L_{4n}(\mathbb{Z})$$
 is the zero map. Now by Theorem \ref{imge} and Corollary \ref{sur}, we get that 
$${\rm Im}(\mathcal{S}^{\mathit t}(M)\stackrel{F_t}{\longrightarrow}\mathcal{S}^{Diff}(M))={\rm Im}(\mathcal{C}(M)\stackrel{F_c}{\longrightarrow}\mathcal{S}^{Diff}(M)).$$
This completes the proof.
 \end{proof}
The above theorem immediately implies the following result.
\begin{corollary}\label{imge1}
	Let $M$ be a closed smooth manifold homotopy equivalent to $\mathbb{H}P^{n}$, $n\geq 2$. If a smooth manifold $N$ is tangential homotopy equivalent to $M$, then $N$ is homeomorphic to $M$. 	
\end{corollary}
As a consequence of Theorem \ref{tanset} and Theorem \ref{main}, we get :
\begin{theorem}\label{struset}
	\indent
  	\begin{itemize}
  \item[(i)] $\mathcal{S}^{\mathit t}(\mathbb{H}P^3)$ contains exactly $2$ equivalence classes.
    \item[(ii)] $\mathcal{S}^{\mathit t}(\mathbb{H} P^4)=\left \{ [(\mathbb{H} P^4\#\Sigma,h_{can})]~~ |~~\Sigma\in \Theta_{16} \right \}$, where $h_{can}:\Hy P^4\#\Sigma \to \Hy P^4$ is the canonical homeomorphism.
      \item[(iii)] $\mathcal{S}^{\mathit t}(\mathbb{H} P^5)$ has exactly $48$ classes of manifolds tangentially homotopy equivalent to $\Hy P^5$.
  \end{itemize}
  \end{theorem}
 The next result follows immediately from Theorem \ref{struset}(ii).
 \begin{corollary}
 Let $M$ be a smooth manifold tangential homotopy equivalent to $\mathbb{H} P^4$. Then there is a homotopy $16$-sphere $\Sigma$ such that $M$ is diffeomorphic to $\mathbb{H} P^4\#\Sigma$.
 \end{corollary}


\begin{thebibliography}{9}
  	\bibitem{AF04}
  	C.~S.~Aravinda,~F.~T.~Farrell, {\em Exotic structures and quaternionic hyperbolic manifolds}, {Algebraic groups and arithmetic,}  Tata Inst. Fund. Res. Stud. Math. (2004), 507--524.

\bibitem{BK} S. Basu, R. Kasilingam, {\em Inertia groups of high dimensional complex projective spaces,} To appear in Algebr. Geom. Topol. 
  	
\bibitem{Bor63} A. Borel, {\em Compact Clifford-Klein forms of symmetric spaces,} Topology (1963) \textbf{2}, 111-122.


\bibitem{Bro72}
  	W.~Browder, {\em Surgery on simply-connected manifolds,}  Springer-Verlag, New York (1972). Ergebnisse der Mathematik und ihrer Grenzgebiete, Band 65.

\bibitem{Bru68}
 		G.~Brumfiel, {\em Differentiable $\mathbb{S}^1$- Actions on Homotopy Spheres,} mimeographed. University of California, Berkeley (1968).	
 	
\bibitem{Bru71}
 		G.~Brumfiel, {\em Homotopy equivalences of almost smooth manifolds,} Comm. Math. Helv. \textbf{46} (1971), 381--407.	
 	
 \bibitem{Bur72}
 D. Burghelea, {\em Free differentiate $\mathbb{S}^1$ and $\mathbb{S}^3$ actions on homotopy spheres,} Ann. Sei. Ecole Norm. Sup. \textbf{(4)} 5 (1972), 183-215.
 	 	
\bibitem{Gle50}
 	A.~Gleason, {\em Spaces with a compact Lie group of transformations,}  Proc. Amer. Math. Soc. \textbf{1} (1950), 35--43.
 	
\bibitem{Her69}
 	D.~N.~Hertz, {\em Ambient surgery and tangential homotopy quaternionic projective spaces,} Trans. Amer. Math. Soc. \textbf{145} (1969), 517--545.		
 	
 		
\bibitem{Hsi66}
 	W.~C.~Hsiang, {\em A note on free differentiable actions of $\mathbb{S}^1$ and $\mathbb{S}^3$ on homotopy spheres,} Ann. of Math. \textbf{83} (1966), 266--272.

 	
\bibitem{Jam76}	
I.~M.~James, {\em Relative Stiefel manifolds,} London Math. Soc. Lecture Notes \textbf{24},  Cambridge Univ., 1976.

\bibitem{Ka12}	Issam H. Kaddoura, {\em De-Suspension of Free $\mathbb{S}^3$-Actions on Homotopy Spheres}, Inter.J.Algebra, Vol. 6, 2012, no. \textbf{20}, 985-994

  \bibitem{Kas16} 
R.~Kasilingam, {\em Inertia Groups and Smooth Structures of $(n-1)$-connected $2n$-manifolds,} Osaka J. Math {\bf 53} (2016),  309--319.
  
\bibitem{Kas15} 
R.~Kasilingam, {\em Farrell-Jones spheres and inertia groups of complex projective spaces,}
Forum Math. {\bf 27} (2015), 3005--3015. 

\bibitem{Kas17}  
R.~Kasilingam, {\em Homotopy Inertia Groups and Tangential Structures,} JP J. Geom. Topol. {\bf 20}  (2017), 91--114.


\bibitem{Kaw68}
K.~Kawakubo, {\em Inertia groups of low dimensional complex projective spaces and some free differentiable actions on spheres I,} Proc. Japan Acad. \textbf{44} (1968), 873--875.


\bibitem{Kaw69} 
K.~Kawakubo, {\em On the inertia groups of homology tori,} J. Math. Soc.
Japan, {\bf 21} (1969), 37-47.

\bibitem{KL66}  N. Kuiper, ~R. Lashof, {\em Microbundles and bundles,} I. Elementary theory, Invent. Math. \textbf{1} (1966) pp. 1-17.


 \bibitem{KM63}
 M.~A.~Kervaire,~ J.~W.~Milnor, {\em Groups of homotopy spheres I,} Ann. of Math. (2) \textbf{77} (1963), 504--537.
	
 \bibitem{KS77}
 R.~C.~Kirby,~L.~C.~Siebenmann, {\em Foundational essays on topological manifolds, smoothings, and triangulations,} Ann. of Math. Stud. \textbf{88} (1977).
 
 \bibitem{KS07}
 L. Kramer,~ S. Stolz, {\em A diffeomorphism classification of manifolds which are like projective planes,} J. Diff. Geom. \textbf{77} (2007), 177–188.

\bibitem{Ko96} 
S.~O.~Kochman, {\em Bordism, Stable homotopy and Adams spectral sequences,} Fields Inst. Monogr.  {\bf 7} (1996).


\bibitem{Ku69} H.T. Ku, A note on Semi-free actions on homotopy spheres, {\em Proc. Amer. Math. Soc.,} \textbf{22} (1969),614-617.


\bibitem{KK70} H.T. Ku, ~M.C. Ku, {\em Free differentiable actions of $\mathbb{S}^1$ and $\mathbb{S}^3$ on homotopy spheres}, Proc. Amer Math.Soc., \textbf{25} (1970),864-869.

\bibitem{KK71} H.T. Ku,~M.C. Ku, {\em Characteristic spheres of free differentiable actions of  $\mathbb{S}^1$ and $\mathbb{S}^3$ on homotopy spheres}, Trans. Amer. Math. Soc., \textbf{156} (1971), 493-504.

\bibitem{LR65} R. Lashof, ~ M. Rothenberg, {\em Microbundles and smoothing,} Topology \textbf{3} (1965) pp. 357-388.

\bibitem{CNL96} 
C.~N.~Lee, {\em Detection of some elements in the stable homotopy groups of spheres}, Math Z. {\bf 222} (1996), 231--246. 

\bibitem{MM79} 
I.~Madsen,~J.~R.~Milgram, {\em The classifying spaces for surgery and cobordism of manifolds,} Ann. of Math. Stud. {\bf 92} (1979).

\bibitem{MTW80}
I.~Madsen,~L.~Taylor,~B.~Williams, {\em Tangential homotopy equivalences}, Comm. Math. Helv. {\bf 55} (1980), 445--484. 

\bibitem{Mos68} R.~E.~Mosher, {\em Some stable homotopy of complex projective space,} Topology {\bf 7} (1968), 179--193.

\bibitem{MTW76}
I.~Madsen,~C.~B.~Thomas,~C.~T.~C.~Wall, {\em The topological spherical space form problem--II existence of free actions}, Topology {\bf 15} (1976), 375--382. 

\bibitem{Mil56} 
J.~Milnor, {\em On manifolds homeomorphic to the 7-sphere,} Ann. of Math. (2) {\bf 64} (1956), 399--405. 


\bibitem{MT68} 
R.~E.~Mosher,~M.~C.~Tangora, {\em Cohomology operations and applications in homotopy theory,} Courier Corporation, 1968.


 \bibitem{Oku02} 
B.~Okun, {\em Exotic smooth structures on nonpositively curved symmetric spaces,} Algebr. Geom.  Topol.  {\bf 2} (2002), 381--389. 

\bibitem{Rav04} 
D.~Ravenel, {\em Complex cobordism and the stable homotopy groups of spheres}, 2nd edition, AMS Chelsea Publishing, 2004.

\bibitem{Sch71}
 R.~Schultz, {\em On the inertia group of a product of spheres,} Trans. Amer. Math. Soc. {\bf 156} (1971), 137--153. 

\bibitem{Sch87} 
R.~Schultz, {\em Homology spheres as stationary sets of circle actions}, Michigan Math. J. {\bf 34} (1987), 183--200.

 \bibitem{Sul67}
 D.~P.~Sullivan, {\em Triangulating and smoothing homotopy equivalences and homeomorphisms,}  The Hauptvermutung Book, K-Monogr. Math. {\bf 1} (1996), 69--103.

 \bibitem{SS73}
 F.~Sigrist,~U.~Suter, {\em Cross-sections of symplectic Stiefel manifolds,} Trans. Amer. Math. Soc. \textbf{184} (1973), 247--259.


\bibitem{Tam62}
I.~Tamura, {\em Sur les sommes connexes de certaines variétés différentiables,} C. R. Acad. Sci. Paris vol. \textbf{255} (1962), 3104--3106.


 		\end{thebibliography}
\end{document}